\documentclass[11pt,reqno]{amsart}
\usepackage{amsmath}
\usepackage{amsthm}
\usepackage{amsfonts}
\usepackage{amssymb}
\usepackage{color}
\usepackage{graphicx}
\usepackage{hyperref}
\usepackage[margin=1.0in]{geometry}
\usepackage{caption}
\usepackage{cancel}
\usepackage{amsaddr}

\newtheorem{lemma}{Lemma}

\newtheorem{thm}{Theorem}

\begin{document}

\title{Spectral asymptotics and Lam\'e spectrum for coupled particles in periodic potentials}
\author{Ki Yeun Kim and Mark Levi$^\ast$ and Jing Zhou$^\dagger$} 
\address{Department of Mathematics, Penn State University}
\thanks{$^\ast$ML (ORCID 0000-0001-6345-9079) gratefully acknowledges support from the NSF grant DMS--0605878.\\
$\dagger$ Corresponding author: JZ (ORCID 0000-0002-9034-8961) Department of Mathematics, Penn State University, University Park, State College, PA 16802, United States. Email: jingzhou@psu.edu}
\maketitle

\begin{abstract}
We make two observations on the motion of  coupled particles in a periodic potential. Coupled pendula, or the space-discretized sine-Gordon equation is an example of this problem. Linearized spectrum of the synchronous motion turns out to have a hidden asymptotic periodicity in its dependence on the energy; this is the gist of the first observation. Our second observation is the discovery of a  special property of the purely sinusoidal potentials: the linearization around the synchronous solution  is equivalent to the classical Lam\`e equation. As a consequence, {\it  all but one instability zones  of the linearized equation collapse to a point for the one-harmonic potentials}.  This provides a new example where Lam\'e's finite zone potential arises in the simplest possible setting.   
\end{abstract}

\section{ {\bf Introduction; the setting}}    
In this paper we study the stability of a ``binary", i.e. of two elastically coupled particles  in a periodic potential $ V $ on line (c.f. Figure~\ref{fig:pendula}). 

\begin{figure}[thb]
 	\captionsetup{format=hang}
	\center{  \includegraphics{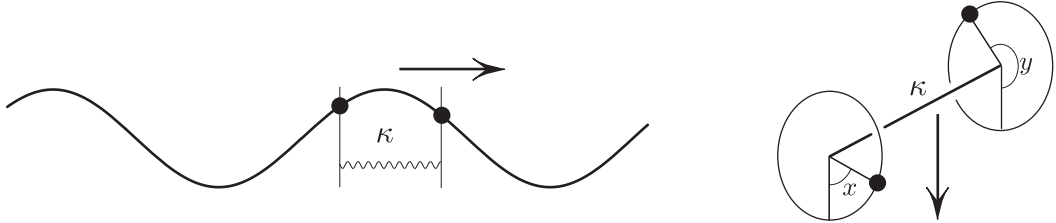}}
	\caption{Elastically coupled particles in a periodic potential $V$. The case $ V= \cos $ corresponds to torsionally coupled pendula.}
	\label{fig:pendula}
\end{figure}

The system is described by    
\begin{align}\label{eq:basicodegeneral}
&\ddot x + V^\prime (x)  = \kappa (y-x)\\
\notag
&\ddot y +V ^\prime (y) = \kappa (x-y). 
\end{align}
The special case of $ V (x) = - \cos x $ can be realized as two coupled pendula (c.f. Figure~\ref{fig:pendula}):
\begin{align}\label{eq:basicode}
&\ddot x + \sin x = \kappa (y-x)\\
\notag
&\ddot y +\sin y = \kappa (x-y).  
\end{align}
We note parenthetically that this system is in fact a discretization of the sine-Gordon equation, which arises naturally in many physical applications such as the Frenkel-Kontorova (F-K) model of electrons in a crystal lattice \cite{BrKi} and the arrays of Josephson junctions\cite{ImSc}.

	 \begin{figure}[thb]
 	\captionsetup{format=hang}
	\center{  \includegraphics{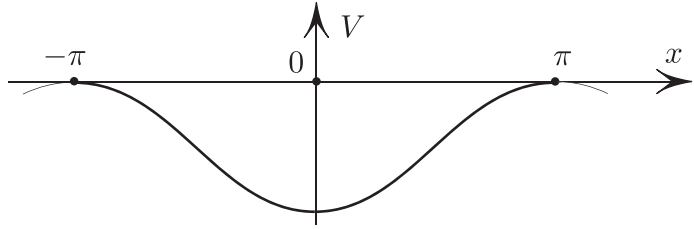}}
	\caption{Periodic potential $V$.}
	\label{fig:V}
\end{figure}

Furthermore, our model is related to the F-K model in the following way: the original F-K model consists of an infinite chain of particles with nearest-neighbor coupling, {\it  in an equilibrium state}.  One can consider instead a  {\it dynamical} Frenkel-Kontorova model: 
\begin{equation}  
	\ddot x_n+ \sin x_n=\frac{\kappa }{2}  (x_{n-1}-2x_n+x_{n +1})   
	\label{eq:dfk} 
\end{equation}  
(here the potential is special: $ V(x) = - \cos x $). The two-particle system (\ref{eq:basicode}) is a special case of   (\ref{eq:dfk}): it governs the evolution of space-periodic solutions of (\ref{eq:dfk}) of period $2$:    
\[
	x_n=x_{n+2} \  \  \hbox{for all }  \  \  n \in {\mathbb Z}.  
\]  
Indeed, substituting    $ x=x_{2n}, \  \    y = x_{2n+1} $ into   (\ref{eq:dfk})  yields   (\ref{eq:basicode}).  
 
We call a solution of   (\ref{eq:basicodegeneral})  {\it  synchronous} if  $ x=y $ for all $t$; the common angle   $ x=y\equiv p $ satisfies the single pendulum equation $ \ddot p+ V ^\prime ( p) = 0 $. 
 Up to  time translation,  solutions are determined by the total energy 
\begin{align}\label{eq:energy}
E = \frac{\dot{p}^2}{2} + V(p); 
 \end{align}
  $ E = 0 $ corresponds to the unstable equilibrium and also to the heterocinic solutions.  We thus think of   $ E$ as the ``excess energy" (or energy deficit if $ E < 0 $).  We study the stability of the synchronous solution $  x=y\equiv p   $, where $p$ satisfies
\begin{equation} 
	   \ddot p + V ^\prime ( p) = 0,
	\label{eq:sync}
\end{equation}     
 as $E$ varies.   
 
 The solution $ p = p(t;E) $ of   (\ref{eq:sync}) with the energy $E$ as in   (\ref{eq:energy}) is defined up to the time shift; we remove this ambiguity  by imposing the the condition  $ p(0, E) = 0 $    throughout the paper. 
 \vskip 0.2 in 
\noindent  {\bf Assumptions on $V$:} We assume throughout the paper that (c.f. Figure \ref{fig:V})
\begin{enumerate} \item $V: {\mathbb R}  \rightarrow {\mathbb R}  $,  $ V\in C^{(2)}( {\mathbb R}  ) $,     $ V(x+ 2 \pi )= V(x) $, 
\item   $V$ achieves  global maximum at $ x = \pi \,\hbox{mod} \,2 \pi $, with 
\begin{equation}
	 \label{eq:Vassumption} 
	V( \pi) = 0 \  \  \hbox{and}  \  \  V ^{\prime\prime} (0)< 0, 
\end{equation}   
\item There are no other global maxima, i.e. $ V(x)< 0 $ for all $ x\not= \pi + 2 \pi n $, $ n \in {\mathbb Z}  $. 
\end{enumerate} 
\vskip 0.1 in 
We note that the canonical  example of coupled pendula (\ref{eq:basicode})  satisfies these assumptions. 

Under these assumptions we make two observations on the coupled particle  in periodic  potential $V$. First, we show that the linearized spectrum of the synchronous motion has a hidden asymptotic periodicity in its dependence on the (logarithm of) energy. Second, we show that in the special case of   sinusoidal potentials the linearization around the synchronous solution is equivalent to the classical Lam\`e equation. We use this fact to prove that sinusoidal potentials have a remarkable property: synchronous motions in such potentials are never hyperbolic except for one interval of energy values. The appearance of Lam\'e's equation as a linearization around periodic motions of  particles in a H\'enon-Heiles   potentials in $ {\mathbb R} ^2 $ has been observed earlier by  Churchill, Pecelli and Rod in their extensive study \cite{CPR}. 

This equation was first introduced by Lam\'e in 1837 in the separation of variables of the Laplace equation in elliptic domains \cite{Lame1837} and later shown to arise  in many situations, for instance in the study of the Korteweg-de Vries equation \cite{Novikov1974}. The double pendula example yields a yet another appearance of the Lam\'e equation in perhaps the most basic setting.

This paper is structured as follows. In Section 2 we describe the infinitely repeating loss and gain of strong stability with the change of energy, and the  asymptotic periodicity of the linearized spectrum as a function of  the logarithm of the energy, in the limit of small energies. In Section 3 we show that for large energies synchronous solutions are   linearly stable the for all periodic (sufficiently smooth) potentials. Finally, in  the last Section 4  we (i) show that the linearization of synchronous solution of coupled pendula is a special case of  Lam\'e's equation,  (ii)  show that only one interval of instability survives, i.e. that sinusoidal potentials are ``exceptionally stable", and (iii)  give  an explicit expression for the interval of unstable energy values.

  \section{{\bf Asymptotic spectral periodicity.}}
In this section we show that the Floquet spectrum of  the synchronous solution changes periodically as the function of the logarithm of the energy, asymptotically for small energies. We first make a topological observation in Section 2.1 that the synchronous solution oscillates between stability and instability zones infinitely many times as the energy $E$ approaches zero. In Section 2.2 we then describe the asymptotically periodic spectral dependence on the energy $E$ for small energies.
  
 \subsection{\bf A topological observation.}
 
 We first note a fact based on a topological argument: as the energy $E$ crosses a neighborhood of $ E = 0$, the synchronous solution    (\ref{eq:sync}) loses and regains strong stability infinitely many times.  To make this statement more precise, consider the linearization of   (\ref{eq:basicode}) around the synchronous solution   (\ref{eq:sync}):     

\begin{align} 
&\ddot{\xi} +  V ^{\prime\prime}  ( p) \xi + \kappa( \xi  -\eta)=0\nonumber \\
&\ddot{\eta}+ V ^{\prime\prime}  ( p)\eta + \kappa(\eta - \xi)=0.\nonumber  
\end{align}
By setting $ u = \xi + \eta $,  $w=\xi-\eta$ we decouple this system into
\begin{align}\label{eq:cml}
&\ddot u+ V ^{\prime\prime} (p) u = 0\\
&\ddot{w}+ (2\kappa+ V ^{\prime\prime} (p)   ) \;w = 0.\label{eq:plame}
\end{align}

The decoupled linear system with periodic coefficients can be written in the Hamiltonian form and thus the spectrum of the associated Floquet matrix is symmetric with respect to the unit circle. Two of these eigenvalues are   $ \lambda_1= \lambda_2 = 1$; these eigenvalues correspond to   (\ref{eq:cml}).  The remaining two eigenvalues $ \lambda_3, \  \lambda_4 $ correspond to  (\ref{eq:plame}). These determine stability of the synchronous solution; we will call this solution  {\it strongly stable} if   $ \lambda_3, \   \lambda_4  $ lie on the unit circle and differ from $ \pm 1 $.

As $ E $ decreases to $0$ (the heteroclinic value), strong stability of the synchronous solution is lost and regained infinitely many times, according to the following theorem.

\begin{thm}
Assume that $V$ satsfies the conditions (1)-(3) above, and let $2\kappa > -V ^{\prime\prime} ( \pi )$. Then  there exists a monotone decreasing sequence of disjoint segments $ [E_{2n}, E _{2n-1}] $ clustering at   $0$: 
\begin{equation} 
	E_1\geq E _2>E _3\geq E _4> \ldots \downarrow 0
	\label{eq:seq}
\end{equation}   
such that for some  $ E \in [E_{2n}, E_{2n-1}] $ the synchronous solution of  (\ref{eq:basicode}) with energy $E$   is not strongly stable, i.e. the eigenvalues
$ \lambda_3 $, $ \lambda_4 = \lambda_3 ^{-1}  $ of the linearization   (\ref{eq:plame})  are real.  Outside these intervals, i.e. for $ E \in(E_{2n+1}, E_{2n})$ and for $ E \in (E_1, \infty ) $ the linearlization is strongly stable. 
\end{thm} 

\begin{proof}
The proof follows the same idea as in \cite{Levi88}.  
The angle $\theta = \arg (w+ i \dot w) $ in the phase plane of   (\ref{eq:plame})  satisfies 
\begin{equation} 
	\dot \theta = - \sin ^2 \theta - (2 \kappa + V ^{\prime\prime}  (p )) \cos ^2 \theta   ,
	\label{eq:thetaeq}
\end{equation} 
where $ p=p(t; E ) $ increases monotonically from $- \pi $ to $  \pi $ over the period 
\begin{equation} 
	T( E ) =\frac{1}{  \sqrt{ 2 } }  \int_{- \pi }^{  \pi } \frac{dx}{  \sqrt{ E-V(x) }}, 
	\label{eq:period}
\end{equation}  
obtained from   (\ref{eq:energy}). 
  For any solution $ \theta $ of   (\ref{eq:thetaeq})  we have $ \theta (T( E )) \rightarrow - \infty $   as $ E \downarrow 0 $ -indeed,     $  T( E ) \rightarrow \infty $ as $ E \downarrow 0 $, while $ \dot \theta $ is bounded away from $0$: by a constant:  
  \[
	\dot \theta < -  \sin ^2 \theta -(2 \kappa+ V ^{\prime\prime}  (p) ) \cos ^2 \theta < - \sin ^2 \theta - 
	(2 \kappa +V ^{\prime\prime}  ( \pi ) \cos^2 \theta <-\min (1, 2 \kappa +V ^{\prime\prime}  (\pi ))<0. 
\]  
    
This implies that the Floquet matrix of   (\ref{eq:plame})  enters and leaves stability domains in $ SL(2, {\mathbb R}  ) $ infinitely many times, Figure~\ref{fig:sl2r}, i.e. there exists an infinite sequence of open and disjoint $ E $--intervals for  which the Floquet matrix lies in one of the two elliptic domains in the symplectic group. 
\end{proof}

\subsection{\bf The normal form.} 

In this section we show the asymptotic periodicity of the spectrum of  the trace of the Floquet matrix of   (\ref{eq:plame})) as $ E $ approaches  $0$.  

\begin{thm} \label{logscale} 
Assume that  $ V   $ satisfies the assumptions stated at the end of the previous section, and introduce  
\begin{equation} 
	\lambda =  \sqrt{ - V ^{\prime\prime} (\pi ) }, 
	\label{eq:lambda}
\end{equation}  
the positive eigenvalue of the saddle $ ( \pi \, \hbox{mod}\, \,2 \pi, 0 ) $ in the phase plane of  $ \ddot x + V ^\prime (x) =0 $. Assume also that $ 2 \kappa + V ^{\prime\prime} (\pi ) \equiv 2 \kappa + \lambda ^2 >0  $ and define the frequency $ \omega $ via\footnote{ $ \omega $ is the frequency of ``internal oscillations" of the ``binary" near the maximum of the potential. }
\begin{equation} 
	 \omega ^2 = 2 \kappa + \lambda ^2    > 0.  
	\label{eq:positive}
\end{equation}   
 There exist constants $ a\geq 2 $ and $ \varphi$ depending on the potential $V$  and on $\kappa$  such that the Floquet matrix   of   (\ref{eq:plame})  satisfies  
 \begin{equation} 
	{\rm tr}\, F_E = a \cos  \biggl( \frac{\omega }{\lambda }  \ln E -\varphi  \biggl)    + o(E ^0 ). 
		\label{eq:normalform}
\end{equation}   
 \end{thm}  
 
  \begin{figure}[thb]
 	\captionsetup{format=hang}
	\center{  \includegraphics{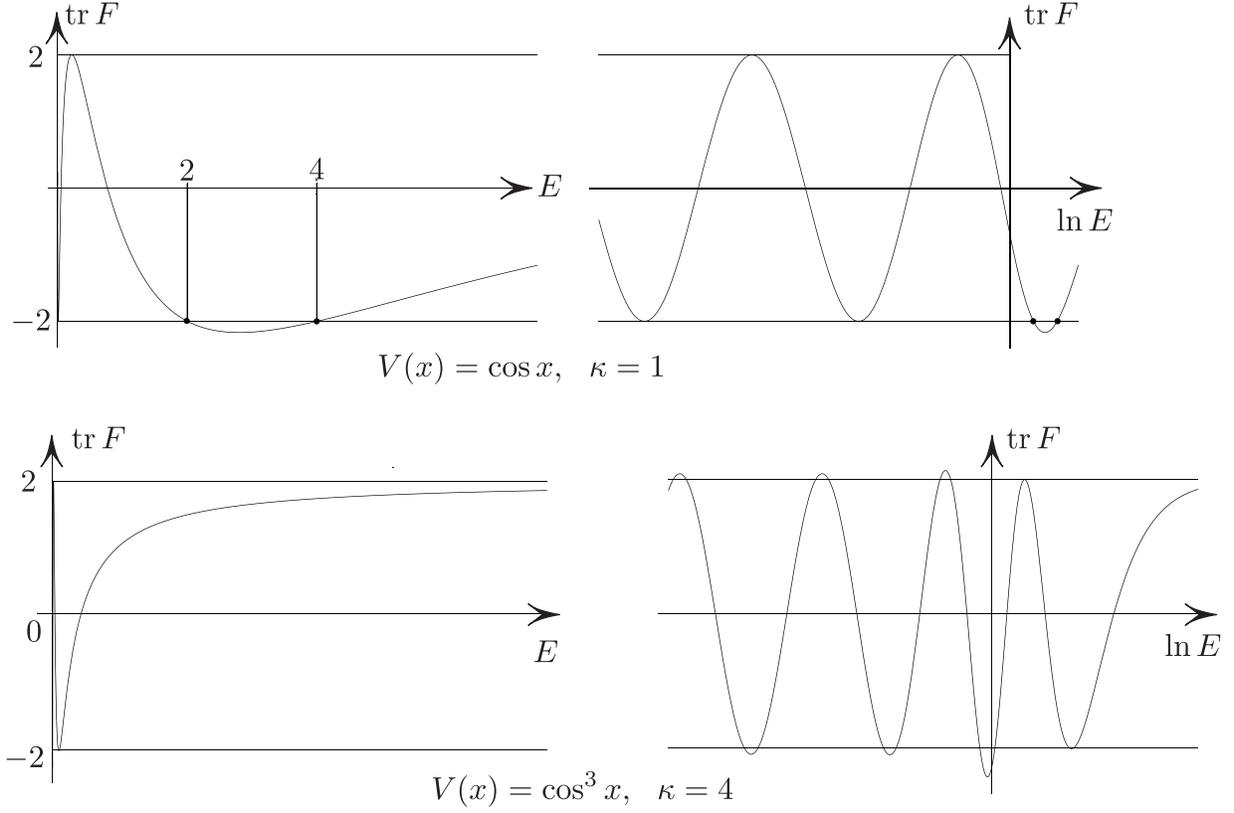}}
	\caption{For $ V(x)= \cos x$ all but one resonance intervals collapse; for $ V(x) = \cos ^3x $ the intervals open up.}
	\label{fig:traces}
\end{figure}  

In Section 2.2.1 we state the key lemmas and give a brief outline of the proof of Theorem \ref{logscale};   the details of the proof are given in Section 2.2.2, with the key lemmas assumed; these  lemmas are proven in Section 2.2.3.

\subsubsection{\bf Key lemmas}

The linearized equation 
\begin{equation} 
	\ddot w+ (2\kappa + V ^{\prime\prime} (p (t; E )) )w= 0  
	\label{eq:lingen}
\end{equation} 
has the periodic coefficient  of period $ T( E ) $ given by    (\ref{eq:period}) - the time it takes for $ p  $ to change from  $ - \pi $ to $\pi$. 

Our goal is to prove   (\ref{eq:normalform})  for the  $ T(E) $ -- advance matrix  $F= F_ E $ of the linear system associated with 
   (\ref{eq:lingen}): 
  \begin{equation}  
   \left\{ \begin{array}{l} 
  \dot w = u \\[3pt] 
   \dot u = - (2\kappa + V ^{\prime\prime} (p (t; E )) w. \end{array} \right.  
  \label{eq:lingens}
\end{equation}  

We state the following three key lemmas, proven later in Section 2.2.3.

 \begin{lemma}\label{T}   
 Assume that $V$ is  $ 2 \pi $-periodic and has a unique non-degenerate maximum at $ x=\pi ({\rm mod}2 \pi ) $, with $ V( \pi ) = 0 $ (as assumed throughout).  There exists a constant    $ K> 0$    depending on $V$   such that for small $ E > 0 $ 
\begin{equation} 
	 T(E)\equiv 2 \tau  (E)= \frac{1}{ \sqrt{ 2 } }  \int_{-\pi }^{\pi } \frac{dx}{ \sqrt{  -V(x) + E}  } = 
	\frac{1 }{  \lambda }  \ln \frac{1}{E} +K+ o( E^0),
	\label{eq:asyint}
\end{equation} 
 where  $ \lambda = \sqrt{ - V ^{\prime\prime} (\pi)} $.
 \end{lemma} 
 
\begin{lemma} [``exponential death"]\label{asyconst}
Assume that the coefficient matrix of the matrix ODE $ \dot C = G(t) C $ decays exponentially in both future and the past: there exist   $ c> 0 $ and $  \lambda  > 0 $ such that 
\begin{equation} 
	\Vert G(t) \Vert \leq c e^{- \lambda  |t|} \  \  \hbox{for all }  \  \  t\in (- \infty , \infty ). 
	\label{eq:zero}
\end{equation}   
 Then there exists a constant matrix $N$ and a constant $b$ depending only on $G$  such that any fundamental solution matrix $ C$  satisfies 
\begin{equation} 
	\Vert  C(t)C ^{-1} (-t)   - N \Vert\leq b e^{- \lambda t} \  \  \hbox{for all}  \  \  t>0.
	\label{eq:transition}
\end{equation}   
In other words, the time-advance map $ C(t)C ^{-1} (-t)  $ from $ -t $ to $t$ approaches $N$ exponentially as $ t \rightarrow \infty $. 
\end{lemma} 

\begin{lemma}\label{proxy}
There exists some constant $ c> 0 $ such that 
\begin{equation} 
	| p(t ) - p_0(t) | \leq c E^ \frac{ \sqrt{ 2 } }{2}     \  \  \hbox{for}  \  \ | t | \leq \tau ( E ), 
	\label{eq:p}
\end{equation}  
where $ p(t) = p(t; E ) $ and $ p_0(t)=p(t; 0) $. \end{lemma}

Based on the three key lemmas whose proofs are presented in Section 2.2.3, the   proof of Theorem \ref{logscale} proceeds in two steps:

 \noindent{\bf Step 1 - outline.}  We replace $ p( t, E ) $ in   (\ref{eq:lingen})  with $ p_0(t)=p( t, 0 ) $, the heteroclinic solution (corresponding to $ E=0 $) which approaches  $\pm \pi $  as $ t \rightarrow \pm  \infty $:
 \begin{equation} 
	\ddot w+ (2\kappa + V ^{\prime\prime} (p_0(t))w= 0  
	\label{eq:lingen0}
\end{equation} 
 and consider first the time advance matrix $ F^0_ E $  for  the modified system  (\ref{eq:lingen0}) but over the  period $ T(E) $  of the unmodified system  (\ref{eq:lingen}).  We note that $ F^0_ E $ depends on $E$ only through $\tau = \tau (E) $, whereas $ F_E $ depends on $E$ in one additional way, namely through the depenence of  coefficient matrix  on $ p(t;E) $. {\it  We will prove  (\ref{eq:normalform}) for $ F_E^0 $.}  A crucial use will be made of the fact that the coefficient in  (\ref{eq:lingen0}) approaches a positive constant at a sufficiently fast exponential rate (c.f. Lemma \ref{asyconst}).

\noindent {\bf Step 2 - outline.} We  will show that  $ \Vert  F_ E  - F_ E ^0 \Vert=o(E^0 ) $; together with Step 1 this would imply   (\ref{eq:normalform}) thus completing the proof of the theorem.\\

In the next Section 2.2.2, we present the details of the proof of Theorem \ref{logscale} based on the three key lemmas stated in this section. The proofs of these key lemmas can be found in Section 2.2.3.

\subsubsection{\bf Proof of Theorem \ref{logscale}}
Based on the three key lemmas in Section 2.2.1, we now prove Theorem \ref{logscale}.

\begin{proof} 
\noindent Following the idea of Step 1 outlined above,  
let us write the system  (\ref{eq:lingen0})  in vector form, splitting the coefficient matrix into a constant part and the part  that decays at  infinity:  
\begin{equation} 
	\dot z = (A+ B(t)) z,  
	\label{eq:lingenvec}
\end{equation}   
where  
\[
	A = \left( \begin{array}{cc} 0 & 1 \\ - \omega^2  & 0 \end{array} \right) , \  \   
\  \  
	B(t) = \left( \begin{array}{cc} 0 & 0 \\ - \lambda ^2 -V ^{\prime\prime}  (p_0(t))& 0 \end{array} \right), 
\]   
where $\omega$ and $\lambda$ were defined in the statement of the theorem. Now $ 0< \pi-|p_0(t)| < ce^{- \lambda | t|} $ for some $ c>0 $ and for all $t$  since 
 $ - \lambda $ is the stable eigenvalue of the saddle in the phase plane of $ \ddot x + V ^\prime (x) = 0 $. Thus $ | V ^{\prime\prime} (\pi ) - 
 V ^{\prime\prime} (p_0(t)) | \leq  ce^{- \lambda | t|}$ and hence 
   the (say) Frobenius norm of $B$ decays at infinity:
\begin{equation} 
	\Vert B(t) \Vert  \leq  c  e^{- \lambda  | t | }  \  \  \hbox{for all}  \  \   t\in {\mathbb R}.
	\label{eq:B}
\end{equation}  

As stated in the outline, we define the  Floquet matrix $ F^0_E$ as the  $T( E ) $--advance map of  the  linearization (\ref{eq:lingenvec}) around the heteroclinic solution, i.e. one corresponding to $ E=0 $,  from time   $ t=- \tau $ to $ t=\tau $, where $ \tau = T( E )/2 $. Letting $X$ be the fundamental solution matrix 
 of   (\ref{eq:lingenvec})  we have  
\begin{equation} 
	F^0_E{=} X(\tau) X ^{-1} (- \tau ) 
	\label{eq:F0}
\end{equation}  

 Indeed, assuming $ X(0)=I $, $ X(- \tau ) $ propagates an initial condition vector from $t=0$ to $t= - \tau $; thus $ X ^{-1} (- \tau )$ propagates from $ - \tau $ to $0$. And $ X( \tau ) $ propagates from $0$ do $\tau$; this yields   (\ref{eq:F0}).  To estimate this product, we strip off the elliptic-rotational part of $X$ by introducing  the matrix function $ C(t) $ via 
\begin{equation} 
	X = e^{At}C. 
	\label{eq:fsm}
\end{equation} 

Substitution    into   (\ref{eq:lingenvec})  shows that $C$  satisfies the ODE 
\[
	\dot C = e^{-At} B(t)  e^{At} C. 
\]  
Note that $ e^{At} $ is an elliptic rotation (since $ \omega ^2 = 2 \kappa + V ^{\prime\prime} (\pi) > 0 $), and thus is bounded, together with its inverse, for all $t$, so that   the coefficient matrix decays at infinity: 
\[
	\Vert e^{-At} B(t)  e^{At} \Vert \leq c e^{- \lambda | t | } ,  
\]  
as follows from   (\ref{eq:B}). 

According to Lemma \ref{asyconst}  this property implies  existence of  a constant matrix $N$ such that    
\[
	C(\tau )  C ^{-1} ( - \tau ) \buildrel{  (\ref{eq:transition}) }\over{=}  N+ O ( e^{- \lambda \tau } );  
\]  
 here and throughout  $ O(f(t)) $ denotes a function whose absolute value is bounded 
  by $ c f(t) $ for all $t$, for some constant $c$   independent of $t$ (and $E$). Substituting   (\ref{eq:fsm}) into   (\ref{eq:F0}) 
  and using the last estimate we obtain
\[
	F^0_E=e^{A\tau}N e^{A\tau} + O ( e^{- \lambda \tau } );
\]  
where we again used the boundednes of $ e^{At} $ for all $ t\in {\mathbb R}  $. 

A simple calculation shows that 
\begin{equation} 
	{\rm tr}\,  (e^{A\tau}N e^{A\tau}) =  (n_{11}+n_{22}) \cos 2 \omega \tau + 
	\biggl( \frac{1}{\omega }n_{21}- \omega n_{12} \biggl)   \sin2 \omega \tau, 
	\label{eq:tr}
\end{equation}   
where $ n_{ij} $ are the elements of $N$, so that 
\begin{equation} 
	{\rm tr}\, F^0_E = a \cos (2 \omega \tau - c_1 ) + O ( e^{- \lambda \tau } ).
	\label{eq:F1}
\end{equation}  

We observe that   $  {\rm det}\; N = 1 $, as follows from    the Hamiltonian character of our linear systems. Using this in   (\ref{eq:tr})   implies that the amplitude $ a\geq 2 $. 
We have 
\[
	2 \omega \tau = \omega T   \buildrel{\rm {Lemma} \ref{T} }\over{=}   
	\omega   
	\biggl( \frac{ 1}{\lambda  } \ln \frac{1}{E } + K+o(E ^0) \biggl) ,  
\]  
so that (\ref{eq:F1})  becomes 
\[
	{\rm tr}\, F^0_E=a \cos \biggl(  \frac{\omega }{\lambda }  \ln  E   - K \biggl)   + o(E ^0 ) . 
\]  
 \marginpar{stopped here}  
In accordance with  Step 2 mentioned before,  to complete the proof of Theorem \ref{logscale} it remains to show that  $ \Vert F_E- F_E^0 \Vert = o ( E ^0) $. 

We show that Lemma \ref{proxy} implies  $ \Vert F_E- F_E^0 \Vert = o ( E ^0) $. Recall that  $F_E$ is expressed in terms of the fundamental solution matrix $ X_E(t) $  of $ \dot z = L(t;E) z $ with 
\[
	L(t;E) = \left( \begin{array}{cc} 0 & 1 \\ -2 \kappa - V ^{\prime\prime} (p(t;E)) & 0\end{array} \right) 
\]  
via $ F_E= X_E( \tau ) X_E ^{-1} ( \tau ) $. Similarly,  $ F_E^0= X_0( \tau ) X_0 ^{-1} ( \tau ) $ (this is just a repetition of     (\ref{eq:F0})). It therefore suffices to prove that 
\begin{equation} 
	\sup_{t\in [0, \tau]} \Vert X_E(t) - X_0(t) \Vert = o ( E^0) . 
	\label{eq:xx}
\end{equation}   
To that end we subtract $ \dot X_0 = L_0 X_0 $ from $ \dot X_E = L_E X_E$ (where we abbreviated $ L(t; E) $ as $ L_E $) and obtain  
\[
	\frac{d}{dt} ( X_E-X_0)= L_0\,(X_E-X_0) + (L_E-L_0)X_E,  
\]  
or 
\[
	X_E(t)-X_0(t) = \int_{0}^{t}  \underbrace{ X_0(t) X_0 ^{-1} ( s)}_{\leq M} \,\underbrace{ (L_E(s)-L_0(s))}_{\leq c E ^{  \sqrt{2}/2} } X_E(s) ds, 
\]  
for all $ t\in [0, \tau ] $, where $M$ is a constant independent of $E$ and in the above estimate  of $ L_E-L_0 $ we have applied Lemma \ref{proxy}. This implies that for all sufficiently small $E$  
\[
	|X_E-X_0 |_{C^0[0, \tau  ]}\leq \tau Mc E ^{\sqrt{2}/2}  \, |X_E |_{C^0[0,  \tau]}\leq c_1 E^ {\sqrt{2}/4}|X_E |_{C^0[0, \tau ]}  
\]  
where we used  $\tau = O( \ln E ^{-1} ) $. This implies  (\ref{eq:xx}), and hence conclude the proof. 
\end{proof}

\subsubsection{\bf Proof of key lemmas}
In this section we prove the key lemmas stated in Section 2.2.1.

\begin{proof}[Proof of Lemma \ref{T}]
Let us replace $V  $ by  its leading Taylor term and consider the resulting change in the integral, considering first the interval $ [0, \pi ] $:  
\begin{equation} 
	\int_{0}^{\pi } \biggl(  \frac{ 1}{ \sqrt{  -V(x) + E}  } - 
	  \frac{ 1 }{ \sqrt{ \frac{1}{2} \lambda ^2 (x-\pi ) ^2 + E}  } \biggl)\;  dx.
	  \label{eq:difint}
\end{equation}  
  As $E \downarrow 0 $, the  integrand converges pointwise on $ [0, \pi ] $  to the function 
 \begin{equation} 
	f (x) = \frac{ 1}{ \sqrt{  -V(x) }  } -  \frac{ \sqrt{ 2 } }{\lambda   (\pi -x)    },  \  \  x\in [0, \pi ], 
	\label{eq:f}
\end{equation}  
defined at $ x=\pi  $ by continuity;  $f$  is {\it  bounded and continuous }   (singularities cancel at $ x=\pi  $,  the key point of the proof),  and convergence is monotone at every $x$ (increasing or decreasing 
 depending on the sign of the Taylor remainder). We conclude that   (\ref{eq:difint})  converges to $ \int_{0}^{\pi} f\;dx $. Now  the integral of the second term in the   integrand of   (\ref{eq:difint})  
 is computed explicitly  as 
 \[
 	 \frac{\sqrt{ 2 }}{ \lambda  }\biggl(\frac{1}{2}  \ln\frac{1}{ E} + 
	  \ln \bigl( 2 \pi  \lambda  \bigl)  \biggl)  + O( E ) ,
\]
and therefore  convergence of    (\ref{eq:difint})  to $ \int_{0}^{\pi} f\;dx $ implies   
\[
	\int_{0}^{\pi }   \frac{ 1}{ \sqrt{  -V(x) + E}  } =  
	\frac{\sqrt{ 2 }}{  \lambda  }\biggl(\frac{1}{2}  \ln\frac{1}{ E} + 
	  \ln \bigl( 2 \pi  \lambda  \bigl)  \biggl) 
	+ \int_{0}^{\pi } f (x) \;dx+O ( E ^0 ).  
\]  
 The integral over $ [- \pi, 0 ] $ is estimated similarly; adding the two estimates yields   (\ref{eq:asyint}), with 
 \[
	K =  \frac{2\sqrt{ 2 }}{ \lambda } \ln \bigl( 2 \pi  \lambda  \bigl)+\int_{-\pi}^{\pi } f(x) \; dx, 
\]  
 where $f$  was defined above in terms of $V$. 
\end{proof}

\begin{proof}[Proof of Lemma \ref{asyconst}]
Since $ C(t) C ^{-1} (t) $ is independent of the choice of the fundamental solution matrix, we lose no  generality by assuming  $ C(0)=I $.  
\[
	\biggl|\frac{d}{dt} \Vert C  \Vert \biggl| \leq  \Vert  \dot C \Vert  = \Vert 	GC\Vert\leq c e^{- \lambda | t | } \Vert    C \Vert; 
\]  
 this proves boundedness: $ \Vert    C(t) \Vert\leq c/ \lambda  $ for all $ t\in {\mathbb R}  $.  Now
 \begin{equation} 
	C(t) = I + \int_{0}^{t} G(s) C(s) \;ds. 
	\label{eq:int}
\end{equation}   
Since $C$ is bounded and $G$ decays exponentially, the improper integral in 
\[
	I + \int_{0}^{\infty} G(s) C(s) \;ds \buildrel{def}\over{=} N_+
\]  
converges; and 
\begin{equation} 
	\Vert   N_+ - C(t)  \Vert  =   \Big\Vert  \int_{t}^{\infty } G(s) C(s) \;ds \Big\Vert \leq \int_{t}^{\infty } c e^{- \lambda  s}\frac{c}{ \lambda } \;ds = 
	\frac{c ^2 }{\lambda^2 } e^{- \lambda t}.
	\label{eq:-}
\end{equation} 

Similarly one shows that  there exists a constant matrix $N_-$ such that  $ C ^{-1} (-t) \rightarrow   N_- $ as $ t \rightarrow \infty$ at the same exponential rate $\lambda$ . 
Together with   (\ref{eq:-})  his implies   (\ref{eq:transition}) with $ N=N_+N_- $  and completes the proof of the lemma. 
\end{proof}

\begin{figure}[thb]
 	\captionsetup{format=hang}
	\center{  \includegraphics{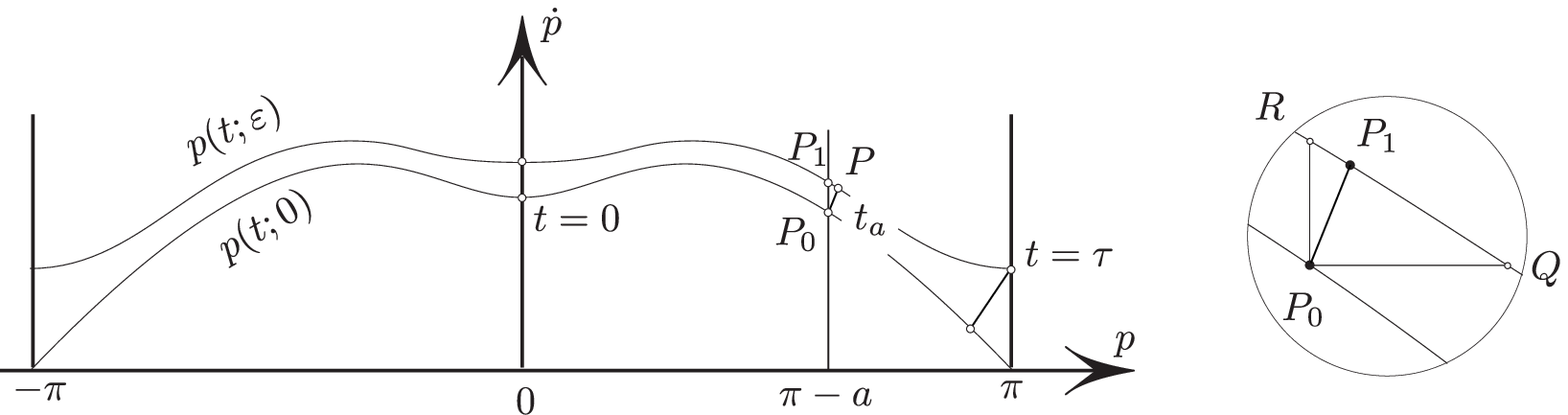}}
	\caption{Proof of   (\ref{eq:p}).}
	\label{fig:p}
\end{figure}

\begin{proof}[Proof of Lemma \ref{proxy}]
 \noindent {\bf  1.}  We first note that $ p(t)>p_0(t) $ for all $ t> 0 $; indeed, $ p(0)=p_0(0) $ and $p$, $ p_0 $ satisfy   
 \begin{equation} 
	\dot p =  \sqrt{ 2(E-V(p)) } , \  \  \dot p _0=  \sqrt{ 2(-V(p_0)) }, 
	\label{eq:odep}
\end{equation}  
and the claim follows from the comparison theorem for the first order ODEs. 

\noindent {\bf 2.}  From $V^\prime ( \pi ) =0$ and $ V^{\prime\prime} ( \pi ) < 0 $ we conclude that $ V ^\prime $ is monotone increasing in some interval $ [\pi -a, \pi ] $, $ a> 0 $.  Define $ t_a $ by $ p_0(t_a) = \pi -a $( Figure~\ref{fig:p}). First we prove that  (\ref{eq:p}) holds for $ t\in [0,t_a] $. 
  Subtracting the second equation in  (\ref{eq:odep})  from the first we obtain, using the mean value theorem:
 \[
	\bigg|\frac{d}{dt} (p-p_0)\bigg|\leq \alpha  (p-p_0)  +\beta E,  
\]  
where $\alpha $ and $\beta $ are the bounds on the partial derivatives of $  \sqrt{ 2(H-V(x)) } $ with respect to $H$ and $x$ over $ x\in [0, \pi -a/2] $, $ H\in [0,E]$; a short calculation gives
\[
	\alpha =\sup_{ x\in [0, \pi -a/2] } \frac{|V ^\prime(x) |}{ \sqrt{ -2V(x)}}  , \  \  \beta  = \sup_{ x\in [0, \pi -a/2] } \frac{1}{ \sqrt{ -2V(x) } } 
\]  
 We conclude that 
 \begin{equation} 
	| p(t)-p_0(t) | \leq\int_{0}^{t_a} e^{\alpha (t- s)} \beta E ds< cE \  \  \hbox{for all}  \  \    t\in [0, t_a], 
	\label{eq:prelim}
\end{equation} 
where $ c =  \beta (e^{\alpha t_a}-1 )/\alpha  $. It remains to prove   (\ref{eq:p}) for the remaining time $ t\in [t_a, \tau ] $. 

\noindent {\bf  3.} Consider the time-shifted solution  $ p_1(t) = p(t- \delta ) $ where $\delta$ is  defined  by $  p(t_a- \delta )= p_0(t_a) = \pi -a$.   Note that $ \delta = O (E) $ as follows from  $ | p(t_a) - p(t_a- \delta ) | =   | p(t_a) - p_0(t_a ) |= O ( E) $ (according to   (\ref{eq:prelim})) and the fact that $\dot p $ is bounded away from $0$ in the region in question.   From $ \delta = O ( E) $ it follows that     $ | p(t) - p_1(t) | < c E $ for  $ t\in [t_a, \tau ] $ since $ \dot p $ is bounded.  

Because of this proximity of $ p $ and $ p_1 $ is suffices to prove  (\ref{eq:p})   with $ p $ replaced by   $ p_1$. To that end we consider the segment $ P_0P_1 $ (Figure~\ref{fig:p}) connecting the two phase points $ P_0=(p_0, \dot p_0 )$ and  $ P_1=(p_1, \dot p_1 )$, and its slope 
 \begin{equation} 
	s   = \frac{\dot p _1- \dot p_0}{p_1-p_0}.   
	\label{eq:slope}
\end{equation}   
We show that $ s(t)> 0 $ for all $ t\in (t_a, \tau ]$ by observing that the segment $ P_0P_1 $ is trapped for all $ t\in (t_a, \tau ]$ in the moving sector $ QP_0 R $ formed by two rays through $ P_0 $, one horizontal and another vertical (Figure~\ref{fig:p}). Indeed, the horizontal shear in the vector field $ \dot x = y, \  \dot y  = - V ^\prime (x) $ shows that $P_1$ cannot leave through the vertical ray. And we now show that $ P_1 $ cannot cross the the horizontal segment because  $ V ^\prime $ in $ [ \pi -a, \pi ] $ is monotonically decreasing. Indeed, assuming the contrary, and let $ t^\ast\in (t_a, \tau ] $ be the first time when $ P_0P_1 $ is horizontal, i.e. when  $ \dot p_1(t^\ast)  = \dot p_0(t^\ast)  $. Then  
\[
	\dot  s  (t^\ast) =\frac{(\ddot p_1 -\ddot p_0 )( p_1 - p_0 )- (\dot p_1 -\dot p_0 ) ^2 }{(p_1-p_0)^2 } =
	\frac{(-V ^\prime (p_1)+ V ^\prime ( p_0) )( p _1- p_0 )  }{(p_1-p_0)^2 } > 0,  
\]  
using the monotonicity of $ V^\prime $ and the fact that $ p_1>p_0$ for $ t> t_a $. But this contradicts $ \dot s( t^\ast) \leq 0 $, and proves that $ s>0 $   for all $ t\in (t_a, \tau] $.
 And this positivity of $s$ implies that $ | p_1-p_0 | = O (  \sqrt{ E } ) $ for $ t\in (t_a, \tau] $. This completes the proof of   (\ref{eq:p}). 
\end{proof}

\section{\bf Stability for Large Energies}
Numerical evidence in Figure~\ref{fig:traces} suggests that the ``binary", i.e. the synchronous solution of   (\ref{eq:basicode}),   is stable for large energies. 
In this section we prove that this is indeed the case for any $C^2$ periodic  potential.

\begin{thm}
Let $V: {\mathbb R}  \rightarrow {\mathbb R}  $     be an arbitrary periodic potential of class $ C^{2}({\mathbb R}) $ (with no further assumption).  For any $\kappa$, there exists $E_{\kappa} \gg 1$ such that the linearized equation (\ref{eq:plame}) is stable for all large energies $E>E_{\kappa}$: the associated Floquet matrix $F_E $ of   (\ref{eq:plame})  is stable for any $E>E_{\kappa}$, where $T=T(E)$ is the period of the synchronized solution $p=p(t,E)$ defined in (\ref{eq:period}).
\end{thm}

\begin{proof}
  Wirting the linearized equation (\ref{eq:plame}) as the first-order linear system   (\ref{eq:lingens}), or  more compactly,  as
   $$\dot{z}=L(t)z$$ 
where $z=(w,\dot{w})$ and 
\[ L(t) =
   \begin{pmatrix}
      0 & 1 \\
      -(2\kappa+V''(p(t))) & 0
   \end{pmatrix}, 
\]
we conclude that 
\[
	| z(T) | \leq | z(0) | e^{l T}, 
\]  
where $ l = \max_{t\in[0,T]}\Vert L(t)\Vert$ is independent of $E$. Here $ \Vert\cdot\Vert $ denotes the matrix norm generated by the Euclidean norm $ | \cdot  |  $.  
Now $T$ is small for large $E$: $T(E)=O(1/\sqrt{E})$, as follows from \ref{eq:period}. Thus for $ t\in [0,T]$ there is little variation in $z$:  
$$z(t) = z(0) + \int_0^t L(s) z(s) ds = z(0) + r_1(t),$$
where $r_1(t) = O(1/\sqrt{E})$. Thus 
\begin{align*}
    z(T) &= z(0) + \int_0^T L(s) z(s) ds \\
         &= z(0) + \int_0^T L(s) z(0) ds + \int_0^T L(s) r_1(s) ds \\ 
         &= (I+\Bar{L}T) z(0) + r_2(T)
\end{align*}
where $\Bar{L}=\frac{1}{T}\int_0^T L(s) ds$ and $r_2(T)=o(1/\sqrt{E})$, and the Floquet matrix is therefore 
$$F_E = I + \Bar{L}T + o(1/\sqrt{E}).$$
To prove that $ F_E $ is stable for large $E$ we compute $ \overline L $ whose form turns to guarantee stability for large $E$.  From \ref{eq:energy} we conclude that 
$p$ grows nearly linearly for large $E$:  
$$\dot{p}(t) = \sqrt{2(E-V)} = \sqrt{2E} + O(1/\sqrt{E}), $$ 
so that 
\begin{align*}
   \overline L= \frac{1}{T} \int_0^T (2\kappa + V''(p(s))) ds
    &= 2\kappa + \frac{1}{T} \int_0^{2\pi} \frac{V''(p)}{\dot{p}} dp\\
    &= 2\kappa + \frac{1}{T} \int_0^{2\pi} \frac{V''(p)}{\sqrt{2E}} dp + o(1/\sqrt{E})\\
    &= 2\kappa + o(1/\sqrt{E}), 
\end{align*}
using periodicity  of $ V ^\prime $ in the last step. 
Therefore the Floquet matrix is of the form 
\[ F_E = 
   \begin{pmatrix}
      1 & T \\
      -2\kappa T & 1
   \end{pmatrix} +  o(1/\sqrt{E}).
\] 

 it is not clear from the last expression whether the stability condition $  | {\rm tr}\, F | <2 $ is satisfied without knowing more about the remainder. Interestingly, this knowledge is not necessary: we will show that the leading term guarantees the absence of real eigenvectors and thus ellipticity.  Absence of real eigenvectors amounts to showing    that  $ F_E u \cdot u^{\perp} \not= 0 $   for any nonzero vector $u=(u_1,u_2)\in\mathbb{R}^2$, which we do now: 
\[
    F_E u \cdot u^{\perp} = T(2 \kappa u_1^2 + u_2^2 )+o(1/\sqrt{E})\geq T\min(1, 2 \kappa )+o(1/\sqrt{E})>0 
\]
 for sufficiently large $E$, since the remainder is small  relative to  $T=T(E)= O (1/ \sqrt{ E } ) $. 
\end{proof}

\section{\bf Collapsed resonances in sinusoidal potentials and Lam\'e's equation}

 So far we described the properties common to  general periodic potentials. Remarkably, in the presence of only  one harmonic: $ V(x) = \kappa  \cos x $  all instability intervals,  except for the first one, collapse to a point:  $ E _{2n}= E_{2n-1} $ for all $ n \geq 2 $.\footnote{under the assumption $ \kappa > \frac{1}{2} $ which is necessary for the existence of an infinite sequence $ E_n $.}    The underlying reason for this collapse of instability intervals  is the fact  that linearization around the synchronous solution is a disguised Lam\'es equation. 
 
 The result of this section implies that sinusoidal potentials are the most stable ones for traveling ``binaries", i.e. that the traveling solutions are never hyperbolically unstable except for one specific interval of energies. 

 In Section 4.1 we show that the linearization of the coupled pendula system is a Lam\'e equation in disguise and consequently in Section 4.2 we show that only one interval of instability survives.

\subsection{The coupled Pendula and the Lam\'e equation}
In this section we reveal the fact that the linearization of the coupled pendula system is a special case of the Lam\'e equation.

We recall that the general Lam\'e equation has the form
\begin{equation} 
	 \ddot{W} + \left[\lambda - n(n+1)k^2 {\rm sn}^2(t, k)\right] W=0  
	\label{eq:lame}
\end{equation}  
where  $n$ is a positive integer and  where ${\rm sn}(t, k)$ is Jacobi's elliptic function.\footnote{Recall the definition of the ``snoidal" function $ {\rm sn}(t,k) $: given
given $t\in {\mathbb R}   $
and $   k \in [0,1) $, one defines $ {\rm sn}(t, k) $ via 
\[
	t = \int_{0}^{{\rm sn}(t, k)} \frac{du}{ \sqrt{(1-u ^2 )(1-k ^2 u^2 ) }}. 
\]  
Equivalently, by substituting $ u = \sin \theta $, one can define $x$ via 
$ t= \int_{0}^{x} \frac{d\theta}{ \sqrt{1-k ^2 \sin ^2 \theta  } } $ and set  $ {\rm sn}(t,k) \buildrel{def}\over{=}  \sin x $. It is this latter form of the definition that we will use.
}

\begin{thm}\label{lameq}  Let  $w$ be a solution of the linearization 
 \begin{equation} 
	 \ddot{w}+ (2\kappa+\sin p  ) \;w = 0
 \end{equation} 
  around the synchronous solution $p$  ($ \ddot p + \sin p = 0 $) 
with the ``energy surplus"  $E$, as in (\ref{eq:energy}), i.e. with $ \dot p ^2 /2 + (1-\cos p ) =E $. 
The rescaled function
\begin{equation} 
	W(\tau) = w(k \tau), \  \  k ^2 =    \frac{2}{2+ E}  
	\label{eq:rescale}
\end{equation} 
satisfies the   Lam\'e equation     corresponding to $ n=1$: 
\begin{equation} 
	W ^{\prime\prime}  + \left[\lambda - 2k^2 {\rm sn}^2(\tau, k)\right] W =0,  	\  \  ^\prime = \frac{d}{d\tau};
	\label{eq:W}
\end{equation}
and with   
\begin{equation} 
	\lambda = (2\kappa +1 ) k ^2 \  \  \hbox{and}  \  \  k ^2 = 2/(2+ E).
	\label{eq:lk}
\end{equation}     
\end{thm} 

\begin{proof}
  To establish a connection between the  coefficients  in  (\ref{eq:plame})  and  (\ref{eq:W})  we express $ \cos p $ in terms of $ {\rm sn}  $. From the energy conservation (\ref{eq:energy}),  and using the trigonometric identity $ 1+ \cos p = 2-2 \sin ^2 (p/2) $ we obtain  the implicit expression for $p$: 
\begin{align*}
t = \int_0^{p/2} \frac{d\theta}{\sqrt{1 + \frac{\epsilon}{2}  -\sin^2 (\theta)}}, 
\end{align*}
 or, setting $\displaystyle k ^2 =\frac{1}{1+\frac{\epsilon}{2}}$, 
\begin{align*}
\frac{t}{k}  &= \int_0^{p/2} \frac{ d\theta}{\sqrt{1 - k^2 \sin^2 \theta}}.
 \end{align*}
 By the definition of $ {\rm sn} $, this gives 
 \[
	{\rm sn} \biggl(  \frac{t}{k} , k \biggl)   = \sin\frac{p(t)}{2} .
\] 
Squaring both sides and using $ 2\sin ^2 (p/2) = 1-\cos p $, we obtain after a short manipulation: 
\begin{equation} 
		  \cos p = 1-2 {\rm sn}^2  \biggl(  \frac{t}{k} , k \biggl).  
	\label{eq:cossn}
\end{equation}   
Substituting this into   (\ref{eq:plame})  we obtain an equivalent form
\[
	\ddot w + \biggl( 2\kappa+1-2 {\rm sn}^2  \biggl(  \frac{t}{k} , k \biggl) \biggl)  w=0. 
\] 
Finally, rescaling $ \tau = t/k $ results in   (\ref{eq:W}) and completes the proof of  Theorem \ref{lameq}. 
\end{proof}

\subsection{Collapse of Instability Intervals}
It is a remarkable and well known fact that for the Lam\'e equation  (\ref{eq:lame})  all but the first $n$ spectral gaps are collapsed: for fixed $k$  and $\lambda$ as the parameter, the trace of the Floquet matrix exceeds $2$ in absolute value  for precisely $n$ 
intervals of $\lambda$ \cite{HHV}. In particular, for the $ n=1 $ case, the first gap is open an all the others are closed. However,  our case  (\ref{eq:W})  $\lambda$ is not an independent parameter but rather both it and  $k$  vary with  the independent parameter $E$; it is thus unclear whether the same statement about gaps in  $E$ holds true. Theorem  \ref{oneinterval}  states that it does. 

\begin{thm} \label{oneinterval}
The synchronous solution of   (\ref{eq:basicode}) (which exists iff $ E>0 $) is linearly unstable if and only if
\begin{equation} 
	E \in (4 \kappa -2, 4 \kappa). 
	\label{eq:interval}
\end{equation}   
For all other values of $E$ the solution is linearly  stable, Figure~\ref{fig:sl2r}. In particular, for $ \kappa >1/2$ the first two terms in the  sequence  $E_m\downarrow 0 $ are  $ E_1= 4 \kappa $,   $ E_2= 4 \kappa -2 $. At all other resonant energies $ E_3, \  E_4, \ldots $ linearization around the synchronous solutions is neutrally stable, i.e. the Floquet matrix of the linearization is  identity. 
\end{thm}  

Since Equation  (\ref{eq:W}) is the $n=1$ case of the general Lam\'e equation, we can study the stability of the linearized system   (\ref{eq:cml})-(\ref{eq:plame}) by using some properties of the Lam\'e equation.

The Lam\'e equation is a special class of the Hill's equations. By the Oscillation Theorem for the Hill operator (Theorem 2.1 in \cite{MaWi}), there exist two sequences of real numbers $\{\lambda_i\}$ and $\{\lambda'_i\}$ tending to infinity such that 
$$\lambda_0<\lambda'_1\le\lambda'_2<\lambda_1\le\lambda_2<\lambda'_3\le\lambda'_4<\cdots, $$
Figure~\ref{fig:sl2r} (left). 
 Here   $\lambda_k$ correspond to the case when $1$ is an eigenvalue of the Floquet matrix, and thus   (\ref{eq:lame}) has a periodic solution of the same period 
 $2K$ as the potential; similarly,    $ \lambda ^\prime_k $ correspond the eigenvalue $-1 $, i.e.  to $2K$-antiperiodic solutions of (\ref{eq:lame}). The half-period  $K=\int_0^{\pi/2}\frac{dt}{\sqrt{1-k^2\sin^2 t}}$. 
\begin{figure}[thb]
 	\captionsetup{format=hang}
	\center{  \includegraphics{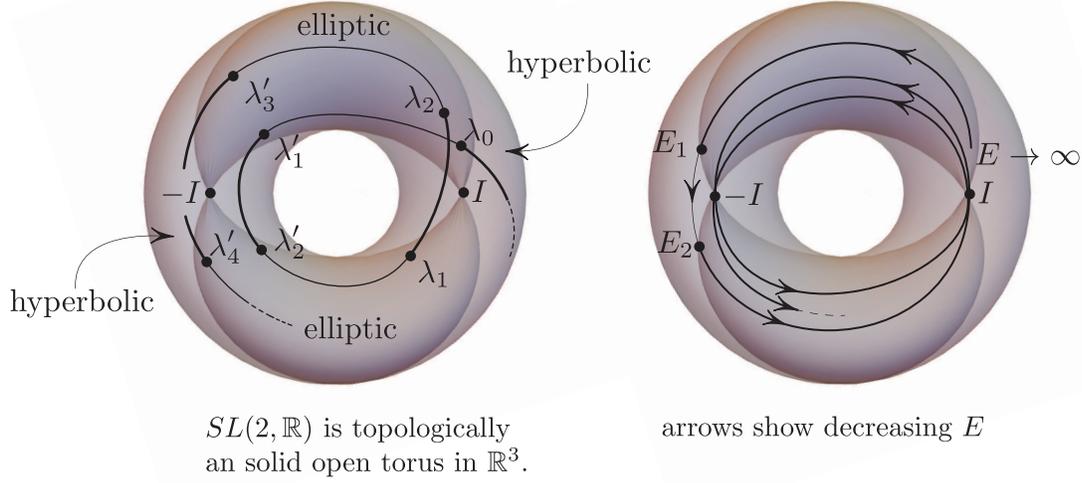}}
	\caption{The path of the Floquet matrix in $ SL(2, {\mathbb R} ) $  for a generic potential as a function of the eigenvalue parameter $\lambda$ (left) and for the Lam\'e potential as the function of the energy $E$   (right). All the $E$--intervals except for $ [E_2, E_1] $ are collapsed.}
	\label{fig:sl2r}
\end{figure}  
 
Floquet matrix of Lam\'e's equation (\ref{eq:lame}) is hyperbolic for $\lambda \in (\lambda_i,\lambda_{i+1})$ and for $ \lambda \in (\lambda'_i,\lambda'_{i+1})$. The instability zone collapses: $ \lambda _i= \lambda _{i+1}$   iff the Floquet matrix is $I$, i.e. iff two two linearly independent solutions of period $2K$ coexist. Simiarly, the collapse $ \lambda'_i=\lambda'_{i+1}$ corresponds to the  Floquet matrix being $ -I $ (Figure~\ref{fig:sl2r}), i.e. to the existence of    two independent  antiperiodic solutions of antiperiod  $2K$.
 
Ince \cite{In} showed that there are at most $n$ open intervals of instability for the Lam\'e equation. Erd\'elyi \cite{Er} proved that Ince's estimate is exact: there are exactly $n$ open intervals of instability for the Lam\'e equation. Later Haese-Hill, Halln\"{a}s and Veselov \cite{HHV} pointed out the position of the instability intervals: the first $n$ instability intervals are open.

The Lam\'e equation  (\ref{eq:lame}) corresponds to the case $ n=1 $ and hence there is only interval of instabililty $(\lambda'_1,\lambda'_2)$ by the foregoing discussion. Now we prove Theorem \ref{oneinterval}.

\begin{proof}[Proof of Theorem \ref{oneinterval}]
We tailor the techniques in \cite{HHV} \cite{MaWi} and, \cite{In} and present the computation of instablity intervals specifically for our case.

To begin our proof, we first discuss the coexistence problem for a more general class called Ince's equation
\begin{equation}\label{Ince}
  (1+a\cos2u)\psi''+b(\sin2u)\psi'+(c+d\cos2u)\psi=0,  
\end{equation}
where $a,b,c,d\in\mathbb{R}$ and $|a|<1$. Then we specialize the analysis to the Lam\'e equation   (\ref{eq:lame}) via the transformation $u=\hbox{am}(\tau,k)$ where am is the Jacobi amplitude defined by $$\frac{du}{d\tau}=\sqrt{1-k^2\sin^2u},$$ and consequently $$a=\frac{k^2}{2-k^2},\ b=-\frac{k^2}{2-k^2},\ c=\frac{2\lambda-n(n+1)k^2}{2-k^2},\ d=\frac{n(n+1)k^2}{2-k^2}.$$

By Lemma 7.3 in \cite{MaWi}, if the Ince's equation (\ref{Ince}) has two linearly independent $\pi$-antiperiodic solutions, then its fundamental solutions $\psi_1,\psi_2$ take the forms 
\begin{equation}\label{funsol}
   \psi_1=\sum_{n=0}^{\infty}A_{2m+1}\cos(2m+1),\ \psi_2=\sum_{n=0}^{\infty}B_{2m+1}\sin(2m+1). 
\end{equation}

Substituting (\ref{funsol}) into (\ref{Ince}), we obtain the following recurrence relations 
\begin{equation}\label{anticos}
    \begin{cases}
       \left(Q(-\frac{1}{2})+\Lambda(\frac{1}{2})\right)A_1 + Q(-\frac{3}{2})A_3=0\\
       Q(m-\frac{1}{2})A_{2m-1} + \Lambda(m+\frac{1}{2})A_{2m+1} + Q(-(m+\frac{3}{2}))A_{2m+3}=0,\ m\ge1
    \end{cases}
\end{equation}
and 
\begin{equation}\label{antisin}
    \begin{cases}
       \left(-Q(-\frac{1}{2})+\Lambda(\frac{1}{2})\right)B_1 + Q(-\frac{3}{2})B_3=0\\
       Q(m-\frac{1}{2})B_{2m-1} + \Lambda(m+\frac{1}{2})B_{2m+1} + Q(-(m+\frac{3}{2}))B_{2m+3}=0,\ m\ge1
    \end{cases}
\end{equation}
where $Q(m)=2am^2-bm-\frac{d}{2}$, $\Lambda(m)=4m^2-c$.

We observe that for the Lam\'e equation (\ref{eq:lame}) we have $Q(n-\frac{1}{2})=0$. Given this observation, by examining the recurrence relations (\ref{anticos})(\ref{antisin}), the Ince's equation (\ref{Ince}) will have two linearly independent $\pi$-antiperiodic solutions if we know one of them has finite order larger than $n$ or infinite order (Theorem 7.3 in \cite{MaWi}), thus closing the instability intervals. The only way to create one but not two linearly independent $\pi$-antiperiodic solution is to find solutions of order smaller than $n$ (Theorem 7.6 in \cite{MaWi}). As a consequence, we consider the following two finite order linear systems 
\begin{equation}\label{Asystem}
    \begin{pmatrix}
       Q(-\frac{1}{2})+\Lambda(\frac{1}{2}) & Q(-\frac{3}{2}) & & \\
       Q(\frac{1}{2}) & \Lambda(\frac{3}{2}) & Q(-\frac{5}{2}) & \\
       & & \cdots & & \\
       & & &Q(n-\frac{3}{2}) &\Lambda(n-\frac{1}{2})
    \end{pmatrix}
    \begin{pmatrix}
       A_1\\
       A_3\\
       \cdots\\
       A_{2n-1}
    \end{pmatrix}
    =
    \begin{pmatrix}
       0\\
       0\\
       0\\
       0
    \end{pmatrix},
\end{equation}
and
\begin{equation}\label{Bsystem}
    \begin{pmatrix}
       -Q(-\frac{1}{2})+\Lambda(\frac{1}{2}) & Q(-\frac{3}{2}) & & \\
       Q(\frac{1}{2}) & \Lambda(\frac{3}{2}) & Q(-\frac{5}{2}) & \\
       & & \cdots & & \\
       & & &Q(n-\frac{3}{2}) &\Lambda(n-\frac{1}{2})
    \end{pmatrix}
    \begin{pmatrix}
       B_1\\
       B_3\\
       \cdots\\
       B_{2n-1}
    \end{pmatrix}
    =
    \begin{pmatrix}
       0\\
       0\\
       0\\
       0
    \end{pmatrix}.
\end{equation}

We can find nontrivial solutions $(A_1,A_3,\cdots,A_{2n-1})$, $(B_1,B_3,\cdots,B_{2n-1})$ to the above linear systems (\ref{Asystem})(\ref{Bsystem}) by making the coefficient matrices singular, i.e. 
\begin{equation}\label{Adet}
    \det\begin{vmatrix}
       Q(-\frac{1}{2})+\Lambda(\frac{1}{2}) & Q(-\frac{3}{2}) & & \\
       Q(\frac{1}{2}) & \Lambda(\frac{3}{2}) & Q(-\frac{5}{2}) & \\
       & & \cdots & & \\
       & & &Q(n-\frac{3}{2}) &\Lambda(n-\frac{1}{2})
    \end{vmatrix}
    =0
\end{equation}
and 
\begin{equation}\label{Bdet}
    \det\begin{vmatrix}
       -Q(-\frac{1}{2})+\Lambda(\frac{1}{2}) & Q(-\frac{3}{2}) & & \\
       Q(\frac{1}{2}) & \Lambda(\frac{3}{2}) & Q(-\frac{5}{2}) & \\
       & & \cdots & & \\
       & & &Q(n-\frac{3}{2}) &\Lambda(n-\frac{1}{2})
    \end{vmatrix}
    =0 .
\end{equation}

Since we only need the result for $n=1$ Lam\'e equation, we present the computation for $n=1$ here.

In our case, $$Q(-\frac{1}{2})=-\frac{k^2}{2-k^2},\ \Lambda(\frac{1}{2})=1-\frac{2\lambda-2k^2}{2-k^2},$$
thus (\ref{Adet}) gives $\lambda'_1=1$ and (\ref{Bdet}) gives $\lambda'_2=1+k^2$. 
We conclude the proof by substituting the relations $k^2=\frac{2}{2+\epsilon}$ and $\lambda=(2\kappa+1)k^2$ from 
Theorem \ref{lameq}.
\end{proof}

\bibliography{references}
\bibliographystyle{plain}

\end{document}